\documentclass[12pt,reqno]{amsart}

\usepackage{amssymb}
\usepackage[all]{xy}

\numberwithin{equation}{section}

\newtheorem{theorem}{Theorem}[section]
\newtheorem{proposition}[theorem]{Proposition}
\newtheorem{lemma}[theorem]{Lemma}

\theoremstyle{definition}

\newtheorem{remark}[theorem]{Remark}

\begin{document}

\baselineskip=15pt

\title[Criterion for Lie algebroid connections]{A criterion for Lie algebroid connections
on a compact Riemann surface}

\author[I. Biswas]{Indranil Biswas}

\address{Department of Mathematics, Shiv Nadar University, NH91, Tehsil Dadri,
Greater Noida, Uttar Pradesh 201314, India}

\email{indranil.biswas@snu.edu.in, indranil29@gmail.com}

\author[P. Kumar]{Pradip Kumar}

\address{Department of Mathematics, Shiv Nadar University, NH91, Tehsil Dadri,
Greater Noida, Uttar Pradesh 201314, India}

\email{Pradip.Kumar@snu.edu.in}

\author[A. Singh]{Anoop Singh}

\address{Department of Mathematical Sciences, Indian Institute of Technology (BHU), Varanasi 221005, India}

\email{anoopsingh.mat@iitbhu.ac.in}



\date{}

\begin{abstract}
Let $X$ be a compact connected Riemann surface and $(V,\, \phi)$ a holomorphic Lie algebroid on $X$ such that the
holomorphic vector bundle $V$ is stable. We give a necessary and sufficient condition on holomorphic
vector bundles $E$ on $X$ to admit a Lie algebroid connection.
\end{abstract}

\maketitle

\section{Introduction}

A Lie algebroid on a complex manifold $M$ is a locally free coherent analytic sheaf $V$ on $M$ equipped with a Lie algebra
structure
$$
[-,\, -] \,\,:\,\, V\otimes_{\mathbb C} V \,\, \longrightarrow\,\, V
$$
together with a holomorphic homomorphism of vector bundles $\phi\, :\, V\, \longrightarrow\, TM$,
which is called the \textit{anchor map} of the Lie algebroid, that satisfies the condition
that $[fs,\, t]\,=\, f[s,\, t]-\phi(t)(f)s$ for all locally defined holomorphic sections $s,\, t$ of $V$ and
all locally defined holomorphic functions $f$ on $M$.
Given a holomorphic vector bundle, or more generally a holomorphic principal bundle, on $M$,
its Atiyah bundle (see \cite{At}) is an example of a Lie algebroid.

Connection on a vector bundle is a fundamental concept in differential geometry. An extension of the classical notion of connection 
on a vector or principal bundle over a differentiable manifold $Y$ can be naturally achieved by substituting the tangent bundle of 
$Y$ with a Lie algebroid over $Y$. These generalizations are called Lie algebroid connections.

A Lie algebroid connection on a holomorphic vector bundle $E$ on $M$ is a first order holomorphic differential operator
$$
D\,\,:\,\, E\,\,\longrightarrow\, \, E\otimes V^*
$$
such that $D(fs) \,=\, fD(s) + s\otimes \phi^*(df)$
for all locally defined holomorphic sections $s$ of $E$ and all locally defined holomorphic
functions $f$ on $M$, where $\phi^*\,:\, T^*M \, \longrightarrow\, V^*$ is the dual of the anchor map $\phi$.
When $V\,=\, TM$ and $\phi\,=\,
{\rm Id}_{TM}$, a Lie algebroid connection (respectively, holomorphic Lie algebroid connection)
is an usual connection (respectively, holomorphic connection).
Many of the classical results in the theory of connections, including the Chern--Weil theory of 
characteristic classes, extend naturally to the context of Lie algebroid connections in the smooth 
($\mathcal{C}^\infty$) category. For an introduction to Lie algebroids and an in-depth treatment of Lie 
algebroid connections in the smooth ($\mathcal{C}^\infty$) category see \cite{KM}, \cite{RF}, \cite{ELW}. 
Holomorphic Lie algebroids over a complex manifold have been discussed in \cite{SC}, \cite{AO}, \cite{To}.

Analogous to the classical setup, Lie algebroid connections always exist on a smooth vector bundle over a differentiable manifold; 
however, this is not the case in the holomorphic set-up. Our aim here is to establish a criterion for the existence of 
holomorphic Lie algebroid connections on a holomorphic vector bundle over a compact connected Riemann surface.

A holomorphic vector bundle is called indecomposable if it is not a direct sum of two vector 
bundles of positive ranks. If we express a holomorphic vector bundle $E$ on a compact 
connected complex manifold, in two different ways, as direct sum of indecomposable vector 
bundles $$ \bigoplus_{i=1}^m A_i\,=\, E\,=\, \bigoplus_{j=1}^n B_j, $$ then $m\,=\, n$, and 
there is a permutation $\sigma$ of $\{1,\, \cdots ,\, m\}$ such that $A_i$ is isomorphic to 
$B_{\sigma(i)}$ \cite[p.~315, Theorem 2(ii)]{At1}. The isomorphism classes of vector bundles 
in $\{A_1,\, \cdots,\, A_m\}$ are called the indecomposable components of $E$. Let $X$ be a 
compact connected Riemann surface. A holomorphic vector bundle $E$ over $X$ admits a 
holomorphic connection if and only if the degree of each indecomposable component of $E$ is 
zero \cite{At}, \cite{We}; this is known as the {\it Atiyah--Weil criterion}. In particular, 
an indecomposable holomorphic vector bundle on $X$ admits a holomorphic connection if and 
only if its degree is zero. It is a natural question to look for a similar criterion for the 
existence of holomorphic Lie algebroid connections on holomorphic vector bundles over $X$.

We prove the following (see Proposition \ref{prop1}):

\begin{proposition}\label{prop0.1}
Let $(V,\, \phi)$ be a Lie algebroid on a compact connected Riemann surface $X$ such that
\begin{itemize}
\item ${\rm rank}(V) \, \geq\, 2$, and

\item the vector bundle $V$ is stable.
\end{itemize}
Then any holomorphic vector bundle $E$ on $X$ admits a Lie algebroid connection.
\end{proposition}

\begin{remark}\label{rem-i}
Note that in the special case where the anchor map $\phi$ is identically zero, a Lie algebroid connection
on $E$ is simply an ${\mathcal O}_X$--linear homomorphism $E\, \longrightarrow\, E\otimes V^*$. It should be
clarified that this homomorphism is allowed to be the zero map. Consequently, when $\phi\,=\, 0$, any holomorphic
vector bundle $E$ has a Lie algebroid connection given by the zero homomorphism from $E$ to $E\otimes V^*$.
\end{remark}

Next, for the case of ${\rm rank}(V)\,=\, 1$, we prove the following (see Lemma \ref{lem2}):

\begin{lemma}\label{lem0.2}
Let $(V,\, \phi)$ be a Lie algebroid on a compact connected Riemann surface $X$ such
that ${\rm rank}(V)\,=\, 1$. Assume that one of the following two holds:
\begin{enumerate}
\item $V\,\not=\, TX$;

\item if $V\,=\, TX$, then $\phi\,=\, 0$.
\end{enumerate}
Then any holomorphic vector bundle $E$ on $X$ admits a Lie algebroid connection.
\end{lemma}

Combining Proposition \ref{prop0.1} and Lemma \ref{lem0.2} with the Atiyah--Weil criterion
we obtain the following (see Theorem \ref{thm1}):

\begin{theorem}\label{thm0.1}
Let $(V,\, \phi)$ be a Lie algebroid on a compact connected Riemann surface $X$ such that the vector
bundle $V$ is stable.
\begin{enumerate}
\item If $\phi$ is not an isomorphism, then any holomorphic vector bundle $E$ on $X$ admits
a Lie algebroid connection.

\item If $\phi$ is an isomorphism, then a holomorphic vector bundle $E$ on $X$ admits
a Lie algebroid connection if and only if the degree of each indecomposable component of $E$ is zero.
\end{enumerate}
\end{theorem}

The curvature of a Lie algebroid connection $D$ on $E$ is a holomorphic section of
$\text{End}(E)\otimes \bigwedge\nolimits^2 V^*$. The Lie algebroid connection $D$ is called
integrable if its curvature vanishes identically. Note that $D$ is integrable if
${\rm rank}(V)\,=\, 1$.

A natural question is find criteria for the existence of integrable Lie algebroid connections
on a given holomorphic vector bundle.

\section{Holomorphic Lie algebroid connections}\label{HLC}

Let $X$ be a compact connected Riemann surface. The holomorphic cotangent and tangent bundles
of $X$ will be denoted by $K_X$ and $TX$ respectively. A Lie algebra structure on a holomorphic vector
bundle $V$ 
is a $\mathbb C$--bilinear pairing given by a sheaf homomorphism
$$
[-,\, -] \,\,:\,\, V\otimes_{\mathbb C} V \,\, \longrightarrow\,\, V
$$
such that $[s,\, t]\,=\, -[t,\, s]$ and $[[s,\, t],\, u]+[[t,\, u],\, s]+[[u,\, s],\, t]\,=\, 0$
for all locally defined holomorphic sections $s,\, t,\, u$ of $V$. The Lie bracket operation on $TX$ gives
the structure of a Lie algebra on it. A Lie algebroid on $X$ is a pair $(V,\, \phi)$, where
\begin{enumerate}
\item $V$ is a holomorphic vector bundle on $X$ equipped with the structure of a Lie algebra,

\item $\phi\, :\, V\, \longrightarrow\, TX$ is an ${\mathcal O}_X$--linear homomorphism such that
\begin{equation}\label{se}
\phi([s,\, t])\,=\, [\phi(s),\, \phi(t)],
\end{equation}
for all locally defined holomorphic sections $s,\, t$ of $V$, and

\item $[fs,\, t]\,=\, f[s,\, t]-\phi(t)(f)s$ for all locally defined holomorphic sections $s,\, t$ of $V$ and
all locally defined holomorphic functions $f$ on $X$.
\end{enumerate}
The above homomorphism $\phi$ is called the \textit{anchor map} of the Lie algebroid.

\begin{remark}\label{rem-s}
The second condition \eqref{se} in the above definition actually follows from the first and the third conditions.
In fact, this is proved by a straightforward computation.
\end{remark}

Let $(V,\, \phi)$ be a Lie algebroid on $X$. We have the dual homomorphism
\begin{equation}\label{e1}
\phi^*\,:\, K_X\, \longrightarrow\, V^*
\end{equation}
of $\phi$. A Lie algebroid connection on a holomorphic vector bundle
$E$ on $X$ is a first order holomorphic differential operator
$$
D\,\,:\,\, E\,\,\longrightarrow\, \, E\otimes V^*
$$
such that
\begin{equation}\label{e-4}
D(fs) \,=\, fD(s) + s\otimes \phi^*(df)
\end{equation}
for all locally defined holomorphic sections $s$ of $E$ and all locally defined holomorphic
functions $f$ on $X$, where $\phi^*$ is the homomorphism in \eqref{e1}.

Consider the following short exact sequence of jet bundles for $E$:
\begin{equation}\label{e0}
0\, 
\longrightarrow\, E\otimes K_X \, \stackrel{\alpha}{\longrightarrow}\, J^1(E) \, 
\stackrel{\beta}{\longrightarrow}\, J^0(E)\,=\,E \, \longrightarrow\, 0.
\end{equation}
We recall that the fiber of $J^1(E)$ over any $x\, \in\, X$ is the space of
all holomorphic sections of $E$ over the first order infinitesimal neighborhood
of $x$. The homomorphism $\beta$ in \eqref{e0} is the restriction map that sends
a section of $E$ over the first order infinitesimal neighborhood of $x$ to the
evaluation of the section at $x$. We have the homomorphism 
\begin{equation}\label{ed}
\Delta\, :\, E\otimes K_X\, \longrightarrow\, (E\otimes V^*)\oplus J^1(E)
\end{equation}
that sends any $e\otimes w\, \in\, (E\otimes K_X)_x$, where $x\, \in\, X$, $e\, \in\, E_x$ and $w\,\in\,
(K_X)_x$, to
$$
(-e\otimes\phi^*_x(w),\,\, \alpha_x(e\otimes w))\,\,\in\,\, (E\otimes V^*)_x\oplus J^1(E)_x,
$$
where $\phi^*$ and $\alpha$ are the homomorphisms constructed in \eqref{e1} and \eqref{e0} respectively. Note
that $\Delta$ in \eqref{ed} is fiberwise injective because $\alpha$ is so. Now define the quotient
\begin{equation}\label{e10}
J^1_V(E) \,\, :=\,\, ((E\otimes V^*)\oplus J^1(E))/\Delta(E\otimes K_X).
\end{equation}
Consider the composition of maps
$$
(E\otimes V^*)\oplus J^1(E)\, \longrightarrow\, J^1(E) \, \stackrel{\beta}{\longrightarrow}\, E,
$$
where $\beta$ is the homomorphism in \eqref{e0} and the first map is the natural projection. This
homomorphism evidently vanishes on $\Delta(E\otimes K_X)\, \subset\, (E\otimes V^*)\oplus J^1(E)$ (see
\eqref{e10}), and hence it produces a surjective homomorphism
\begin{equation}\label{eb}
\beta'\,\,:\,\, J^1_V(E)\,\,=\,\,((E\otimes V^*)\oplus J^1(E))/\Delta(E\otimes K_X)
\, \,\longrightarrow\,\, E.
\end{equation}
Let
\begin{equation}\label{ea}
\alpha'\,\,:\,\, E\otimes V^*\,\, \longrightarrow\,\, J^1_V(E)
\end{equation}
be the composition of maps
$$
E\otimes V^*\, \hookrightarrow\, (E\otimes V^*)\oplus J^1(E)\, \longrightarrow\, ((E\otimes V^*)\oplus 
J^1(E))/\Delta(E\otimes K_X)\,=\, J^1_V(E),
$$
where $(E\otimes V^*)\oplus J^1(E)\, \longrightarrow\, J^1_V(E)$ is the quotient
map (see \eqref{e10}). Consider the composition of maps
\begin{equation}\label{ea2}
\beta'\circ\alpha'\, \, :\,\, E\otimes V^*\,\, \longrightarrow\,\, E,
\end{equation}
where $\alpha'$ and $\beta'$ are constructed in \eqref{ea} and \eqref{eb} respectively.
Note that the following composition of homomorphisms vanishes identically:
$$
E\otimes V^*\, \hookrightarrow \, (E\otimes V^*)\oplus J^1(E) \, \longrightarrow\,
J^1(E) \, \longrightarrow\, E,
$$
where $E\otimes V^*\, \hookrightarrow \, (E\otimes V^*)\oplus J^1(E)$ is the natural
inclusion map and $(E\otimes V^*)\oplus J^1(E) \, \longrightarrow\,
J^1(E)$ is the natural projection. Consequently, the map $\beta'\circ\alpha'$ in \eqref{ea2}
vanishes identically. Next it is straightforward to check that the image of $\alpha'$ in
\eqref{ea} is precisely the kernel of $\beta'$ in \eqref{eb}. Consequently, we have
a short exact sequence
\begin{equation}\label{e-1}
0\, \longrightarrow\, E\otimes V^* \, \stackrel{\alpha'}{\longrightarrow}\, 
J^1_V(E) \, \stackrel{\beta'}{\longrightarrow}\, E \, \longrightarrow\, 0
\end{equation}
of holomorphic vector bundles on $X$.

We will show that giving a Lie algebroid connection on $E$ is equivalent to giving a
holomorphic splitting of the short exact sequence in \eqref{e-1}.

First, let $D\, :\, E\, \longrightarrow\, E\otimes V^*$ be a Lie algebroid connection on $E$.
Since $D$ is a holomorphic differential operator of order $1$ from $E$ to $E\otimes V^*$, it
is given by a holomorphic section
$$
\Psi_D\,\, \in\,\, H^0(X,\, E\otimes V^* \otimes J^1(E)^*)\,\,=\,\,
H^0(X,\, \text{Hom}(J^1(E),\, E\otimes V^*)).
$$
Let $\widetilde{\Psi}_D\, :\, (E\otimes V^*)\oplus J^1(E)\, \longrightarrow\, E\otimes V^*$
be the homomorphism defined by $(v,\, w) \, \longmapsto\, v+ \Psi_D (w)$. The composition
of homomorphisms
$$
E\otimes K_X \, \stackrel{\Delta}{\longrightarrow}\, (E\otimes V^*)\oplus J^1(E)\,
\xrightarrow{\,\,\,\widetilde{\Psi}_D\,\,\,}\, E\otimes V^*,
$$
where $\Delta$ is defined in \eqref{ed}, vanishes identically. Consequently, this
composition of homomorphisms produces a homomorphism
$$
J^1_V(E) \,=\, ((E\otimes V^*)\oplus J^1(E))/\Delta(E\otimes K_X)
\,\longrightarrow\, E\otimes V^*
$$
(see \eqref{e10}); let $\widetilde{\Psi}'_D\, :\, J^1_V(E) \,\longrightarrow\, E\otimes V^*$
denote this homomorphism. Since $D$ satisfies \eqref{e-4}, it follows that
$\widetilde{\Psi}'_D$ has the following property: The composition of maps $\widetilde{\Psi}'_D
\circ\alpha'$, where $\alpha'$ is constructed in \eqref{e-1}, coincides with
${\rm Id}_{E\otimes V^*}$. Therefore, $\widetilde{\Psi}'_D$ produces a holomorphic splitting
of the short exact sequence in \eqref{e-1}. Thus, from any Lie algebroid connection on $E$
we have a holomorphic splitting of the short exact sequence in \eqref{e-1}.

To prove the converse, let
$$
\delta\,\,:\,\, J^1_V(E)\,\, \longrightarrow\, \, E\otimes V^*
$$
be a holomorphic homomorphism giving a splitting
of the short exact sequence in \eqref{e-1}, meaning
\begin{equation}\label{ed2}
\delta\circ\alpha'\,=\, {\rm Id}_{E\otimes V^*},
\end{equation}
where $\alpha'$ is the homomorphism in \eqref{e-1}. Now consider
the following composition of homomorphisms:
$$
J^1(E)\,\, \hookrightarrow\,\, (E\otimes V^*)\oplus J^1(E) \,\,\longrightarrow\,\,
((E\otimes V^*)\oplus J^1(E))/\Delta(E\otimes K_X)
$$
$$
=\,\, J^1_V(E) \,\,\stackrel{\delta}{\longrightarrow}\,\, E\otimes V^*,
$$
where $(E\otimes V^*)\oplus J^1(E) \,\,\longrightarrow\,\,
((E\otimes V^*)\oplus J^1(E))/\Delta(E\otimes K_X)$ is the quotient map; let
$$
\widetilde{\delta}\,\,:\,\, J^1(E)\,\,\longrightarrow\,\,E\otimes V^*
$$
denote this composition of maps. Note that $\widetilde{\delta}$ is a holomorphic
differential operator of order one from $E$ to $E\otimes V^*$. From \eqref{ed2}
it follows that $\widetilde{\delta}$ satisfies \eqref{e-4}. In other words,
$\widetilde{\delta}$ is a Lie algebroid connection on $E$.

Therefore, the Lie algebroid connections on $E$ are precisely the holomorphic splittings
of the short exact sequence in \eqref{e-1}.

\begin{remark}\label{rem2}
In \cite{CV} and \cite{To}, a construction similar to the one in \eqref{e-1}
was done. Moreover, in \cite{To}, a similar criterion as above was established.
\end{remark}

We will reformulate the above description of the Lie algebroid connections on $E$.

Tensoring the exact sequence in \eqref{e-1} with $E^*$, we get the short exact sequence
\begin{equation}\label{f1}
0\, \longrightarrow\, E^*\otimes E\otimes V^*\,=\, 
\text{End}(E)\otimes V^* \, \stackrel{\alpha''}{\longrightarrow}\, \text{Hom}(E,\, J^1_V(E)) \,
\stackrel{\beta''}{\longrightarrow}\, \text{End}(E) \, \longrightarrow\, 0.
\end{equation}
Consider the trivial line subbundle ${\mathcal O}_X\otimes_{\mathbb C} {\rm Id}_E\,\subset\, \text{End}(E)$, and
define the subbundle
\begin{equation}\label{f-2}
{\mathcal A}_V(E)\,\,:=\,\, (\beta'')^{-1}({\mathcal O}_X\otimes_{\mathbb C} {\rm Id}_E)\,\,
\subset\, \, \text{Hom}(E,\, J^1_V(E)),
\end{equation}
where $\beta''$ is the homomorphism in \eqref{f1}. So from \eqref{f1} we have the short exact sequence
\begin{equation}\label{f2}
0\, \longrightarrow\,\text{End}(E)\otimes V^* \, \stackrel{\widehat\alpha}{\longrightarrow}\,
{\mathcal A}_V(E) \, \stackrel{\widehat\beta}{\longrightarrow}\, {\mathcal O}_X \, \longrightarrow\, 0,
\end{equation}
where $\widehat\alpha$ and $\widehat\beta$ are given by $\alpha''$ and $\beta''$ respectively.

\begin{lemma}\label{lem1}
The Lie algebroid connections on $E$ are precisely the holomorphic splittings
of the short exact sequence in \eqref{f2}.
\end{lemma}

\begin{proof}
We saw that the Lie algebroid connections on $E$ are precisely the holomorphic splittings
of the short exact sequence in \eqref{e-1}. There is a natural bijection between the
holomorphic splittings
of the short exact sequence in \eqref{e-1} and the holomorphic splittings
of the short exact sequence in \eqref{f2}. Indeed, for any holomorphic homomorphism
$$
\varphi\, :\, E\, \longrightarrow\, J^1_V(E)
$$
giving a holomorphic splitting of \eqref{e-1}, consider the homomorphism
$$
\varphi'\, :\, {\mathcal O}_X \, \longrightarrow\, \text{Hom}(E,\, J^1_V(E))
$$
that sends any locally defined holomorphic function $f$ on $X$ to the locally defined section
of $\text{Hom}(E,\, J^1_V(E))$ that maps any locally defined section $s$ of $E$ to $f\cdot \varphi(s)$.
Note that $$\varphi'({\mathcal O}_X)\, \subset\, {\mathcal A}_V(E)\, \subset\,\text{Hom}(E,\, J^1_V(E)),$$
because $(\beta''\circ\varphi')(f)\,=\, f\cdot \text{Id}_E$, for any locally defined holomorphic
function $f$ on $X$ (the map $\beta''$ is defined in \eqref{f1}). 
Moreover, $\varphi'$ gives a holomorphic splittings of the short exact sequence in \eqref{f2}.
Thus, from any holomorphic splitting of the short exact sequence in \eqref{e-1} we
have a holomorphic splitting of the short exact sequence in \eqref{f2}.

Conversely, for any homomorphism
$$
\psi\,:\, {\mathcal O}_X \, \longrightarrow\, {\mathcal A}_V(E)
$$
giving a holomorphic splitting of \eqref{f2}, consider the homomorphism
$$
\psi_1\, :\, E\, \longrightarrow\, J^1_V(E)
$$
defined by $e\, \longmapsto\, \psi_x(1_x)(e)\, \in\, J^1_V(E)_x$ for all
$x\, \in\, X$ and $e\, \in\, E_x$, where $1_x\,\in\, ({\mathcal O}_X)_x\,=\, {\mathbb C}$
is the image of $1\, \in\,{\mathbb C}$. It is straightforward to check that
$\psi_1$ gives a holomorphic splitting of \eqref{e-1}.
\end{proof}

\section{A criterion for Lie algebroid connections}

For a holomorphic vector bundle $E$ of positive rank on $X$, the number
$$\mu(E)\,:=\, \frac{\text{degree}(E)}{\text{rank}(E)}\,
\in\, {\mathbb Q}$$ is called the \textit{slope} of $E$. A
holomorphic vector bundle $E$ of positive rank on $X$ is called \textit{stable}
(respectively, \textit{semistable}) if $\mu(F) \, <\, \mu(E)$ (respectively, $\mu(F) \,
\leq\, \mu(E)$) for all holomorphic subbundles $F\, \subsetneq\, E$ of positive rank.
A semistable vector bundle $E$ is called \textit{polystable} if it is a direct sum of
stable vector bundles. (See \cite{HL}.)

Let $E$ and $E'$ be semistable vector bundles on $X$ with $\mu(E)\, >\, \mu(E')$. Then
it can be shown that
\begin{equation}\label{ez}
H^0(X,\, \text{Hom}(E,\, E'))\,=\, 0.
\end{equation}
To prove this, if $f\, :\, E\, \longrightarrow\, E'$ is a nonzero homomorphism over $X$, then we
have $\mu(f(E))\, \geq\, \mu(E)$ (as $E$ is semistable and $f(E)$ is a quotient of it),
and $\mu(f(E))\, \leq\, \mu(E')$ (as $E'$ is semistable and $f(E)$ is a subsheaf of it).
These two give $\mu(E)\, \leq\, \mu(E')$, which contradicts the given condition that
$\mu(E)\, >\, \mu(E')$.

We will consider Lie algebroids $(V,\, \phi)$ such that the underlying vector bundle $V$ is
stable. Any line bundle is stable, in particular, $TX$ is stable, so holomorphic connections
are special cases of such Lie algebroid connections.

\begin{proposition}\label{prop1}
Let $(V,\, \phi)$ be a Lie algebroid on $X$ such that
\begin{itemize}
\item ${\rm rank}(V) \, \geq\, 2$, and

\item the vector bundle $V$ is stable.
\end{itemize}
Then any holomorphic vector bundle $E$ on $X$ admits a Lie algebroid connection.
\end{proposition}

\begin{proof}
Let $g$ denote the genus of $X$. First assume that
\begin{equation}\label{e2}
\mu(V)\,=\, \frac{{\rm degree}(V)}{{\rm rank}(V)}\,\, \geq\,\, 2(1-g) \,=\, \text{degree}(TX).
\end{equation}
Then it can be shown that $H^0(X,\, V^*\otimes TX)\,=\, H^0(X,\, {\rm Hom}(V,\, TX))\,=\, 0$.
Indeed, if $f\, :\, V\, \longrightarrow\, TX$ is a nonzero homomorphism, then from
\eqref{e2} it follows that $\mu(\text{kernel}(f))\, \geq\, \mu(V^*)$, because
$\text{degree}(V)\,=\, \text{degree}(\text{kernel}(f))+\text{degree}(TX)$. In that case,
the subsheaf $\text{kernel}(f)\, \subset\, V$ contradicts the stability condition
for $V$ (note that $\text{rank}(0\, <\, \text{kernel}(f))\, <\, \text{rank}(V)$).
In particular, we have $\phi\,=\, 0$.

When $\phi\,=\, 0$, a Lie algebroid connection on $E$ is an element of
$H^0(X,\, \text{End}(E)\otimes V^*)$ (a generalized Higgs field on $E$; a pair consisting of
$E$ and an element of $H^0(X,\, \text{End}(E)\otimes V^*)$ is also known as a Hitchin pair).
Any $E$ admits such a generalized Higgs field; for example, $H^0(X,\, \text{End}(E)\otimes V^*)$
has the zero element. Thus, $E$ admits a Lie algebroid connection.

Now assume that
\begin{equation}\label{e3}
\frac{{\rm degree}(V)}{{\rm rank}(V)}\,\, < \,\, 2(1-g) \,=\, \text{degree}(TX).
\end{equation}

Recall from Lemma \ref{lem1} that the Lie algebroid connections on $E$ are precisely the 
holomorphic splittings of the short exact sequence in \eqref{f2}. Let
\begin{equation}\label{e4}
\tau\,\,\in\,\, H^1(X,\, \text{End}(E)\otimes V^*)\,\,=\,\, H^1(X,\, \text{Hom}(E,\, E\otimes V^*)) 
\end{equation}
be the obstruction class for the holomorphic splitting of \eqref{f2}. By Serre duality,
\begin{equation}\label{e5}
\tau\,\,\in\,\, H^0(X,\, \text{End}(E)\otimes V\otimes K_X)^*\,=\,
H^0(X,\, \text{Hom}(E,\, E\otimes V\otimes K_X))^*.
\end{equation}

First consider the case where the vector bundle $E$ is semistable.

A holomorphic vector bundle $E$ of rank $r$ on $X$ is polystable if and only if the projective
bundle ${\mathbb P}(E)$ is given by a homomorphism from the fundamental group of $X$ to
$\text{PU}(r)$ \cite{NS}. From this characterization of polystable bundles it follows immediately
that the tensor product $E\otimes E'$ of two polystable vector bundles
$E$ and $E'$ is again polystable.
Indeed, if ${\mathbb P}(E)$ is given by $\rho\, :\, \pi_1(X,\, x_0)\, \longrightarrow\,
\text{PU}(r)$ and ${\mathbb P}(E')$ is given by $\rho'\, :\, \pi_1(X,\, x_0)\, \longrightarrow\,
\text{PU}(r')$, then ${\mathbb P}(E\otimes E')$ is given by $\rho\otimes\rho'\, :\,
\pi_1(X,\, x_0)\, \longrightarrow\, \text{PU}(rr')$.

A holomorphic vector bundle $E$ on $X$ is semistable if and only if $E$ admits a filtration
of holomorphic subbundles such that each successive quotient is polystable of same slope. In
view of this, from the above observation that the tensor product
of two polystable vector bundles is again polystable it follows that
the tensor product of two semistable vector bundles is again semistable. It may be
mentioned that the tensor product of two stable vector bundles is not stable in general.

Since $V$ is stable and $E$ is semistable, we conclude that $E\otimes V\otimes K_X$ is semistable. Now from
\eqref{e3} it follows that $\text{degree}(V\otimes K_X)\, <\, 0$. Hence we have
$$
\frac{\text{degree}(E\otimes V\otimes K_X)}{\text{rank}(E\otimes V\otimes K_X)}\,=\,
\frac{\text{degree}(E)}{\text{rank}(E)} +\frac{\text{degree}(V\otimes K_X)}{\text{rank}(V\otimes K_X)}\,
< \, \frac{\text{degree}(E)}{\text{rank}(E)}.
$$
This implies that $H^0(X,\, \text{Hom}(E,\, E\otimes V\otimes K_X))\,=\, 0$ (see \eqref{ez});
recall
that both $E$ and $E\otimes V\otimes K_X$ are semistable. Hence $\tau$ in \eqref{e5} satisfies
the equation $\tau\,=\, 0$. Consequently, the short exact sequence in \eqref{f2} splits
holomorphically. Now from Lemma \ref{lem1} it follows that $E$ admits a Lie algebroid connection.

Next consider the case where the vector bundle $E$ is \textit{ not } semistable. Let
\begin{equation}\label{e6}
0\,=\, F_0\, \subsetneq\, F_1\, \subsetneq\, F_2 \, \subsetneq \, \cdots \, \subsetneq \,
F_{\ell-1} \, \subsetneq \, F_\ell\,=\, E
\end{equation}
be the Harder--Narasimhan filtration of $E$ (see \cite[Section 1.3]{HL}). We have $\ell\, \geq\,
2$ because $E$ is not semistable. We recall that $F_1$ is the unique maximal semistable subbundle
of $E$ of maximal slope. For $2\, \leq\, j\, \leq\, \ell-1$, the quotient bundle $F_j/F_{j-1}$
is the unique maximal semistable subbundle of $E/F_{j-1}$ of maximal slope.

Take any $\theta\, \in\, H^0(X,\, \text{Hom}(E,\, E\otimes V\otimes K_X))$. We will show that
\begin{equation}\label{e7}
\theta(F_i)\,\, \subset\, F_{i-1}\otimes V\otimes K_X
\end{equation}
for all $1\, \leq\, i\, \leq\, \ell$ (see \eqref{e6}).

Since $V$ is stable, and $F_j/F_{j-1}$ is semistable for all $1\, \leq\, j\, \leq\, \ell$, we know
that the vector bundle $(F_j/F_{j-1})\otimes V\otimes K_X$ is semistable. Next note that
$$
\frac{\text{degree}((F_j/F_{j-1})\otimes V\otimes K_X)}{\text{rank}((F_j/F_{j-1})\otimes V\otimes K_X)}
\,=\,\frac{\text{degree}(F_j/F_{j-1})}{\text{rank}(F_j/F_{j-1})} +
\frac{\text{degree}(V\otimes K_X)}{\text{rank}(V\otimes K_X)}\, <\,
\frac{\text{degree}(F_j/F_{j-1})}{\text{rank}(F_j/F_{j-1})}
$$
(the above inequality is a consequence of \eqref{e3}). Hence we have
\begin{equation}\label{e8}
\frac{\text{degree}((F_j/F_{j-1})\otimes V\otimes K_X)}{\text{rank}((F_j/F_{j-1})\otimes V\otimes K_X)}
\, <\, \frac{\text{degree}(F_j/F_{j-1})}{\text{rank}(F_j/F_{j-1})} \, \leq \,
\frac{\text{degree}(F_k/F_{k-1})}{\text{rank}(F_k/F_{k-1})}
\end{equation}
for all $k\, \leq\, j$ and all $1\, \leq\, j\, \leq\, \ell$; the second
inequality is a part of the properties of a Harder--Narasimhan filtration. From \eqref{e8}
it follows that
\begin{equation}\label{e8a}
H^0(X,\, \text{Hom}(F_k/F_{k-1},\, (F_j/F_{j-1})\otimes V\otimes K_X))\,=\, 0
\end{equation}
for all $k\, \leq\, j$ and all $1\, \leq\, j\, \leq\, \ell$ (see
\eqref{ez}).

We will deduce \eqref{e7} from \eqref{e8a} using an inductive argument.

First consider $\theta(F_1)$. From \eqref{e8a} we know that the following
composition of homomorphisms vanishes identically:
$$
F_1 \, \stackrel{\theta}{\longrightarrow}\, E\otimes V\otimes K_X\, \xrightarrow{\,\,\, q_\ell\otimes
{\rm Id}_{V\otimes K_X}\,\,\,}\, (E/F_{\ell-1})\otimes V\otimes K_X,
$$
where $q_\ell\, :\, E\, \longrightarrow\, E/E_{\ell-1}$ is the quotient map. Consequently,
we have $\theta(F_1)\, \subset\, F_{\ell-1}\otimes V\otimes K_X$. Assume that
\begin{equation}\label{k1}
\theta(F_1)\, \subset\, F_j\otimes V\otimes K_X
\end{equation}
with $j\, \geq\, 1$. We will show that
\begin{equation}\label{k2}
\theta(F_1)\, \subset\, F_{j-1}\otimes V\otimes K_X.
\end{equation}

In view of \eqref{k1}, we have the following composition of homomorphisms:
$$
F_1 \, \stackrel{\theta}{\longrightarrow}\, F_j\otimes V\otimes K_X\, \xrightarrow{\,\,\, q_j\otimes
{\rm Id}_{V\otimes K_X}\,\,\,}\, (F_j/F_{j-1})\otimes V\otimes K_X,
$$
where $q_j\, :\, F_j\, \longrightarrow\, F_j/E_{j-1}$ is the quotient map. From
\eqref{e8a} we know that this composition of homomorphisms vanishes
identically. This proves \eqref{k2}. Now inductively we conclude that that
\begin{equation}\label{k2a}
\theta(F_1)\,\,=\,\, 0.
\end{equation}

Next assume that we have a $1\, \leq\, k\, \leq\, \ell-1$ such that
\begin{equation}\label{k3}
\theta(F_i)\, \subset\, F_{i-1}\otimes V\otimes K_X
\end{equation}
for all $i\, \leq\, k$. We will show that
\begin{equation}\label{k4}
\theta(F_{k+1})\, \subset\, F_{k}\otimes V\otimes K_X.
\end{equation}
In \eqref{k1}, \eqref{k2}, replace $E$ by $E/F_k$ and replace the filtration in
\eqref{e6} by the filtration
$$
0\,=\, F_k/F_k\, \subsetneq\, F_{k+1}/F_k\, \subsetneq\, F_{k+2}/F_k
\, \subsetneq \, \cdots \, \subsetneq \, F_{\ell-1}/F_k \, \subsetneq \, F_\ell/F_k\,=\, E/F_k.
$$
{}From \eqref{k3} we know that $\theta$ induces a homomorphism
$$\widehat{\theta}\, :\, E/F_k\, \longrightarrow\, (E/F_k)\otimes V\otimes K_X.$$ 
Now from \eqref{k2a} it follows that $\widehat{\theta}(F_{k+1}/F_k)\,=\, 0$. This implies
that \eqref{k4} holds.

Therefore, inductively it is deduced that \eqref{e7} holds.

Consider the filtration of $E$ in \eqref{e6}. It produces a filtration of holomorphic subbundles
$$
0\,=\, J^1(F_0)\, \subsetneq\, J^1(F_1)\, \subsetneq\, J^1(F_2) \, \subsetneq \, \cdots \, \subsetneq \,
J^1(F_{\ell-1}) \, \subsetneq \, J^1(F_\ell)\,=\, J^1(E),
$$
which, in turn, produces a filtration of holomorphic subbundles
\begin{equation}\label{e9}
0\,=\, J^1_V(F_0)\, \subsetneq\, J^1_V(F_1)\, \subsetneq\, J^1_V(F_2) \, \subsetneq \, \cdots \, \subsetneq \,
J^1_V(F_{\ell-1}) \, \subsetneq \, J^1_V(F_\ell)\,=\, J^1_V(E),
\end{equation}
where $J^1_V(F_i)$ is constructed as in \eqref{e10} by substituting $F_i$ in place of $E$ in \eqref{e10}.

Let
$$\text{End}^F(E)\,\, \subset\,\, \text{End}(E)$$
be the holomorphic subbundle defined by the sheaf of endomorphisms preserving the filtration
in \eqref{e6}. So a section $\gamma\, \in\, \Gamma(U,\, \text{End}(E))$ is a section
of $\text{End}^F(E)$ over $U$ if and only if $\gamma(F_i\big\vert_U)\, \subset\, F_i\big\vert_U$ for
all $0\, \leq\, i\, \leq\, \ell$. Let
$$
{\mathcal A}^F_V(E)\,\, \subset\,\, {\mathcal A}_V(E)
$$
be the holomorphic subbundle (see \eqref{f-2}) defined as follows: A section $\gamma\, \in\,
\Gamma(U,\, {\mathcal A}_V(E))$ is a section of ${\mathcal A}^F_V(E)$ if the homomorphism
$$\gamma\,:\, E\big\vert_U\, \longrightarrow\, J^1_V(E)\big\vert_U$$
(see \eqref{f-2}) sends the subbundle $F_i\big\vert_U\,\subset\, E\big\vert_U$ to the subbundle
$J^1_V(F_i)\big\vert_U\,\subset\, J^1_V(E)\big\vert_U$ in \eqref{e9} for all $1\, \leq\, i\, \leq\,
\ell$. Now from \eqref{f2} we have the following commutative diagram
\begin{equation}\label{e11}
\begin{matrix}
0 & \longrightarrow & \text{End}^F(E)\otimes V^* & \stackrel{\widetilde\alpha}{\longrightarrow} &
{\mathcal A}^F_V(E) & \stackrel{\widetilde\beta}{\longrightarrow} & {\mathcal O}_X & \longrightarrow & 0\\
&&\,\,\,\,\Big\downarrow\Phi &&\Big\downarrow &&\Big\Vert\\
0 & \longrightarrow & \text{End}(E)\otimes V^* & \stackrel{\widehat\alpha}{\longrightarrow} &
{\mathcal A}_V(E) & \stackrel{\widehat\beta}{\longrightarrow} & {\mathcal O}_X & \longrightarrow & 0,
\end{matrix}
\end{equation}
whose rows are exact, the vertical maps are injective (they are the natural
inclusion maps) and $\widetilde\alpha$ (respectively,
$\widetilde\beta$) is the restriction of $\widehat\alpha$ (respectively, $\widehat\beta$).

Let
$$
\Phi_*\,\,:\,\, H^1(X,\, \text{End}^F(E)\otimes V^*)\,\,\longrightarrow\,\, H^1(X,\, \text{End}(E)\otimes V^*)
$$
be the homomorphism of cohomologies
induced by $\Phi$ in \eqref{e11}. Let $$\tau_0\,\,\in\,\, H^1(X,\, \text{End}^F(E)\otimes V^*)$$
be the obstruction class for the holomorphic splitting of the top exact sequence in \eqref{e11}. From
\eqref{e11} it follows that 
\begin{equation}\label{e12}
\Phi_*(\tau_0)\,\,=\,\, \tau,
\end{equation}
where $\tau$ is the cohomology class in \eqref{e4}. To see \eqref{e12}, recall that $\tau$ (respectively,
$\tau_0$) is the image of $1\, \in\, H^0(X,\, {\mathcal O}_X)$ under the homomorphism $\varpi_1\,
:\, H^0(X,\, {\mathcal O}_X)\, \longrightarrow\, H^1(X,\, \text{End}(E)\otimes V^*)$ (respectively,
$\varpi_2\, :\, H^0(X,\, {\mathcal O}_X)\, \longrightarrow\, H^1(X,\, \text{End}^F(E)\otimes V^*)$)
in the long exact sequence of cohomologies associated to the bottom (respectively, top) exact sequence
in \eqref{e11}. Now \eqref{e12} is a consequence of the commutativity of the following diagram:
$$
\begin{matrix}
H^0(X,\, {\mathcal O}_X) & \xrightarrow{\,\,\,\varpi_2\,\,\,} & H^1(X,\, \text{End}^F(E)\otimes V^*)\\
\,\,\, \Big\downarrow{\rm Id} && \,\,\, \Big\downarrow \Phi_*\\
H^0(X,\, {\mathcal O}_X) & \xrightarrow{\,\,\,\varpi_1\,\,\,} & H^1(X,\, \text{End}(E)\otimes V^*)
\end{matrix}
$$
where $\varpi_1$ and $\varpi_2$ are above maps.

Let $tr\, :\, \text{End}(E)\otimes \text{End}(E)\, \longrightarrow\, {\mathcal O}_X$ be the pairing
given by the trace map. So for $v,\, w\, \in\, \Gamma(U,\, \text{End}(E)\big\vert_U)$, we have
$$
tr(v\otimes w) \,=\, \text{trace}(v\circ w)\, \in\, \Gamma(U,\, {\mathcal O}_U).
$$
Now, if $v$ preserves the filtration in \eqref{e6} (meaning $v(s)\, \subset\, \Gamma(U,\, F_i)$
for all $s\, \in\, \Gamma(U,\, F_i)$ and all $1\, \leq\, i\, \leq\, \ell$) and $w$ is
nilpotent with respect to the filtration in \eqref{e6} (meaning $v(s)\, \subset\, \Gamma(U,\, F_{i-1})$
for all $s\, \in\, \Gamma(U,\, F_i)$ and all $1\, \leq\, i\, \leq\, \ell$), then
$v\circ w$ is also nilpotent with respect to the filtration in \eqref{e6}. In that case, we have
\begin{equation}\label{e13}
tr(v\otimes w) \,=\, 0.
\end{equation}

In view of \eqref{e7} and \eqref{e12}, using \eqref{e13} we conclude that $\tau$ in \eqref{e5}
satisfies the equation $\tau\,=\, 0$. Therefore, $\tau$ in \eqref{e4} vanishes.
This implies that the short exact sequence in \eqref{f2} splits holomorphically.
Now from Lemma \ref{lem1} it follows that $E$ admits a Lie algebroid connection.
\end{proof}

\begin{remark}\label{rem1}
We note that Proposition \ref{prop1} is proved in \cite{AO} under the assumption that
the vector bundle $E$ is semistable; see \cite[Corollary 3.14]{AO}.
\end{remark}

\begin{lemma}\label{lem2}
Let $(V,\, \phi)$ be a Lie algebroid on $X$ such that
${\rm rank}(V)\,=\, 1$. Assume that one of the following two holds:
\begin{enumerate}
\item $V\,\not=\, TX$;

\item if $V\,=\, TX$, then $\phi\,=\, 0$.
\end{enumerate}
Then any holomorphic vector bundle $E$ on $X$ admits a Lie algebroid connection.
\end{lemma}

\begin{proof}
Let $g$ be the genus of $X$. First assume that
$$
\text{degree}(V)\, \, \geq\,\, \text{degree}(TX)\,\,=\,\, 2(1-g).
$$
If $\text{degree}(V)\, > \, \text{degree}(TX)$, then $H^0(X,\, \text{Hom}(V,\, TX))\,=\, 0$,
and if $\text{degree}(V)\,= \, \text{degree}(TX)$, then any homomorphism
$V\, \longrightarrow\, TX$ is either zero or an isomorphism.
Thus the given condition in the lemma that $\phi\,=\, 0$ if $V\,=\, TX$ implies
that $\phi\,=\, 0$. As mentioned in the proof of Proposition \ref{prop1}, when $\phi\,=\, 0$,
A Lie algebroid connection on $E$ is simply an element of $H^0(X,\, \text{End}(E)\otimes V^*)$.
Since $H^0(X,\, \text{End}(E)\otimes V^*)$ is nonempty, for example, it has the zero element,
we conclude that $E$ admits a Lie algebroid connection.

Next assume that
\begin{equation}\label{e14}
\text{degree}(V)\, \, < \,\, \text{degree}(TX).
\end{equation}
If $\phi\,=\, 0$, then again the zero homomorphism
$E\,\,\longrightarrow\, \, E\otimes V^*$ is a Lie algebroid connection on $E$. So we
assume that the homomorphism $\phi$ is nonzero.

Let
\begin{equation}\label{e15}
{\mathbb D}\,\,=\,\, {\rm Div}(\phi)
\end{equation}
be the divisor of the section $\phi\, \in\, H^0(X,\, V^*\otimes TX)$. We have
${\rm degree}({\mathbb D})\,=\, 2(1-g)-\text{degree}(V) \, >\, 0$.

Let
$$
\tau\,\,\in\,\, H^1(X,\, \text{End}(E)\otimes V^*)
$$
be the obstruction class for the holomorphic splitting of \eqref{f2}. By Serre duality,
$$
\tau\,\,\in\,\, H^0(X,\, \text{End}(E)\otimes {\mathcal O}_X(-{\mathbb D}))^*,
$$
where $\mathbb D$ is the effective divisor in \eqref{e15}.

Take any nonzero section
$$
\theta\, \in\, H^0(X,\, \text{End}(E)\otimes {\mathcal O}_X(-{\mathbb D}))
\setminus \{0\}\, \subset\, H^0(X,\, \text{End}(E)) \setminus \{0\}.
$$
For any $j\, \geq\, 1$, the section $\theta^j\, \in\, H^0(X,\, \text{End}(E))$ vanishes on
$\mathbb D$, and hence ${\rm trace}(\theta^j)\,=\, 0$ on $\mathbb D$. Since any holomorphic
function on $X$ is a constant one, this implies that ${\rm trace}(\theta^j)\,=\, 0$ on $X$
(recall that ${\mathbb D}\, \not=\, 0$). Consequently, $\theta$ is a nilpotent endomorphism of $E$.

For $1\, \leq\, k\, \, < \, {\rm rank}(E)$, consider the coherent analytic subsheaf
$$
\text{kernel}(\theta^k)\,\, \subset\,\, E.
$$
Note that the quotient $E/\text{kernel}(\theta^k)$ is a subsheaf of $E$ by the homomorphism
$\theta^k$. This implies that $E/\text{kernel}(\theta^k)$ is torsionfree. As a consequence of
it, $\text{kernel}(\theta^k)$ is actually a subbundle of $E$. This subbundle $\text{kernel}(\theta^k)$
of $E$ will be denoted by $W_k$, for notational convenience.

So we have filtration of holomorphic subbundles of $E$
\begin{equation}\label{e16}
0\,=:\, W_0\, \subset\, W_1\, \subset\, W_2 \, \subset\, \cdots \, \subset\, W_{r-1}\, \subset\, W_r\,:=\, E,
\end{equation}
where $r\,=\, {\rm rank}(E)$. The homomorphism $\theta$ is evidently nilpotent with respect to the
filtration in \eqref{e16}.

The filtration in \eqref{e16} produces a filtration of
holomorphic subbundles
\begin{equation}\label{e16b}
J^1_V(W_1)\, \subset\, J^1_V(W_2) \, \subset\, \cdots \, \subset\, J^1_V(W_{r-1})\, \subset\,
J^1_V(W_r)\,:=\, J^1_V(E)
\end{equation}
(see \eqref{e10} for the construction of $J^1_V(W_j)$). Let
$$\text{End}^F(E) \,\, \subset\,\, \text{End}(E)$$
be the holomorphic subbundle defined by the sheaf of
endomorphisms of $E$ preserving the filtration in \eqref{e16}. Let
$$
{\mathcal A}^F_V(E)\,\, \subset\,\, {\mathcal A}_V(E)
$$
be the holomorphic subbundle (see \eqref{f-2}) defined as follows: A section $\gamma\, \in\,
\Gamma(U,\, {\mathcal A}_V(E))$ is a section of ${\mathcal A}^F_V(E)\big\vert_U$ if the homomorphism
$$\gamma\,:\, E\big\vert_U\, \longrightarrow\, J^1_V(E)\big\vert_U$$
(see \eqref{f-2}) sends the subbundle $W_i\big\vert_U\,\subset\, E\big\vert_U$ in \eqref{e16}
to the subbundle $J^1_V(W_i)\big\vert_U\,\subset\, J^1_V(E)\big\vert_U$
in \eqref{e16b} for all $1\, \leq\, i\, \leq\, r$. Now from
\eqref{f2} we have the following commutative diagram
\begin{equation}\label{e17}
\begin{matrix}
0 & \longrightarrow & \text{End}^F(E)\otimes V^* & \longrightarrow &
{\mathcal A}^F_V(E) & \longrightarrow & {\mathcal O}_X & \longrightarrow & 0\\
&&\,\,\,\,\Big\downarrow\Psi &&\Big\downarrow &&\Big\Vert\\
0 & \longrightarrow & \text{End}(E)\otimes V^* & \stackrel{\widehat\alpha}{\longrightarrow} &
{\mathcal A}_V(E) & \stackrel{\widehat\beta}{\longrightarrow} & {\mathcal O}_X & \longrightarrow & 0.
\end{matrix}
\end{equation}
Let
$$
\Psi_*\,\,:\,\, H^1(X,\, \text{End}^F(E)\otimes V^*)\,\,\longrightarrow\,\, H^1(X,\, \text{End}(E)\otimes V^*)
$$
be the homomorphism of cohomologies
induced by $\Psi$ in \eqref{e17}. Let $$\tau_0\,\,\in\,\, H^1(X,\, \text{End}^F(E)\otimes V^*)$$
be the obstruction class for the holomorphic splitting of the top exact sequence in \eqref{e17}. From
\eqref{e17} it follows immediately that 
\begin{equation}\label{e18}
\Psi_*(\tau_0)\,\,=\,\, \tau,
\end{equation}
where $\tau$ is the cohomology class in \eqref{e4} (its proof is same as that of \eqref{e12}).

Since $\theta$ is nilpotent with respect to the filtration in \eqref{e16}, from \eqref{e18} it
follows --- using \eqref{e13} --- that $\tau(\theta)\,=\, 0$. Hence we have $\tau\,=\, 0$. Now
using Lemma \ref{lem1} we conclude that $E$ admits a Lie algebroid connection.
\end{proof}

\begin{theorem}\label{thm1}
Let $(V,\, \phi)$ be a Lie algebroid on $X$ such that the vector bundle $V$ is stable.
\begin{enumerate}
\item If $\phi$ is not an isomorphism, then any holomorphic vector bundle $E$ on $X$ admits
a Lie algebroid connection.

\item If $\phi$ is an isomorphism, then a holomorphic vector bundle $E$ on $X$ admits
a Lie algebroid connection if and only if the degree of each indecomposable component of $E$ is zero.
\end{enumerate}
\end{theorem}

\begin{proof}
Assume that $\phi$ is not an isomorphism. If ${\rm rank}(V)\, =\, 1$, then from
Lemma \ref{lem2} we know that any holomorphic vector bundle $E$ on $X$ admits
a Lie algebroid connection. If ${\rm rank}(V)\, \geq\, 2$, then Proposition \ref{prop1}
says that any holomorphic vector bundle $E$ on $X$ admits a Lie algebroid connection.

Next assume that $\phi$ is an isomorphism. In this case a Lie algebroid connection on $E$ is
a usual holomorphic connection on $E$ \cite{At}. Then a theorem of Weil and Atiyah says that
a holomorphic vector bundle $E$ on $X$ admits
a holomorphic connection if and only if the degree of each indecomposable component of $E$ is zero
(see \cite[p.~203, Theorem 10]{At}).
\end{proof}

\section*{Acknowledgements}

We are very grateful to the two referees for detailed comments to improve the exposition.
The first-named author is partially supported by J. C. Bose Fellowship (JBR/2023/000003).
The third named author is partially supported by SERB SRG Grant SRG/2023/001006.

\section*{Mandatory declarations}

There is no conflict of interests regarding this manuscript. No funding was received for it.
All authors contributed equally. No data were generated or used.

%%%%%%%%%%%%%%%%%%%%%%%%%%%%%%%%%%%%%%%%%%%%%%%%%%%%%%%%%%%%%%%%%%%%%%%%%%%%%%%%%%%%%%%

\end{document}